\documentclass[10pt]{amsart}
\usepackage{amssymb,MnSymbol,times}
\usepackage{amsthm,amsmath,cite}
\usepackage{epic,eepic}

\title[Nonstandard $n$-distances]{Nonstandard $n$-distances based on certain geometric constructions}\thanks{Corresponding author: Jean-Luc Marichal is with the Mathematics Research Unit, University of Luxembourg, Maison du Nombre, 6, avenue de la Fonte, L-4364 Esch-sur-Alzette, Luxembourg. Email: jean-luc.marichal[at]uni.lu}

\author{Gergely Kiss}
\address{Alfred Renyi Institute of Mathematics, Hungarian Academy of Science, HU-1053 Budapest, Re\'altanoda u.\ 13-15, Hungary}
\email{kigergo57[at]gmail.com}

\author{Jean-Luc Marichal}
\address{Mathematics Research Unit, University of Luxembourg, Maison du Nombre, 6, avenue de la Fonte, L-4364 Esch-sur-Alzette, Luxembourg}
\email{jean-luc.marichal[at]uni.lu}

\date{Revised, January 5, 2022}

\begin{document}

\theoremstyle{plain}
\newtheorem{theorem}{Theorem}[section]
\newtheorem{lemma}[theorem]{Lemma}
\newtheorem{proposition}[theorem]{Proposition}
\newtheorem{corollary}[theorem]{Corollary}
\newtheorem{fact}[theorem]{Fact}
\newtheorem{conjecture}[theorem]{Conjecture}
\newtheorem*{main}{Main Theorem}

\theoremstyle{definition}
\newtheorem{definition}[theorem]{Definition}
\newtheorem{example}[theorem]{Example}
\newtheorem{algorithm}{Algorithm}

\theoremstyle{remark}
\newtheorem{remark}{Remark}
\newtheorem{claim}{Claim}

\newcommand{\N}{\mathbb{N}}
\newcommand{\R}{\mathbb{R}}
\newcommand{\Cdot}{\boldsymbol{\cdot}}

\begin{abstract}
The concept of $n$-distance was recently introduced to generalize the classical definition of distance to functions of $n$ arguments. In this paper we investigate this concept through a number of examples based on certain geometrical constructions. In particular, our study shows to which extent the computation of the best constant associated with an $n$-distance may sometimes be difficult and tricky. It also reveals that two important graph theoretical concepts, namely the total length of the Euclidean Steiner tree and the total length of the minimal spanning tree constructed on $n$ points, are instances of $n$-distances.
\end{abstract}



\maketitle

\section{Introduction}

Let $X$ be a set containing at least two elements. Let also $n\geq 2$ be an integer and set $\R_+=\left[0,+\infty\right[$. Recall that an \emph{$n$-distance} \cite{KisMarTeh16} on $X$ is a map $d\colon X^n\to\R_+$ that satisfies the following three conditions:
\begin{itemize}
\item[(i)] $d(x_1,\ldots,x_n)=0$ if and only if $x_1=\cdots =x_n$,
\item[(ii)] $d$ is a symmetric function, i.e., invariant under any permutation of its arguments,
\item[(iii)] $d$ satisfies the \emph{simplex inequality} (\emph{triangle inequality} if $n=2$), i.e.,
$$
d(x_1,\ldots,x_n) ~\leq ~\sum_{i=1}^nd(x_1,\ldots,x_n)_i^z{\,},\qquad x_1,\ldots,x_n,z\in X,
$$
where $d(x_1,\ldots,x_n)_i^z$ stands for the function obtained from $d(x_1,\ldots,x_n)$ by setting its $i$th variable to $z$.
\end{itemize}
The \emph{best constant} associated with an $n$-distance $d$ on $X$ is the number $K^*_n\in\left]0,1\right]$ defined by
$$
K^*_n ~=~ \sup_{\textstyle{x_1,\ldots,x_n,z\in X\atop |\{x_1,\ldots,x_n\}|{\,}\geq{\,} 2}}\frac{d(x_1,\ldots,x_n)}{\sum_{i=1}^nd(x_1,\ldots,x_n)_i^z}{\,}.
$$
Thus defined, $K^*_n$ is the infimum of the set of numbers $K_n\in\left]0,1\right]$ for which the condition
$$
d(x_1,\ldots,x_n)~\leq ~ K_n\sum_{i=1}^nd(x_1,\ldots,x_n)_i^z{\,},\qquad x_1,\ldots,x_n,z\in X,
$$
holds.

To give an example, recall that the \emph{cardinality based $n$-distance} \cite{KisMarTeh16,KisMarTeh18} is defined on $X$ by
$$
d(x_1,\ldots,x_n) ~=~ |\{x_1,\ldots,x_n\}|-1.
$$
Its associated best constant is $K^*_n=(n-1)^{-1}$ and it is attained, e.g., when $x_1\neq x_2$ and $x_2=\cdots =x_n=z$.

Although various definitions of $n$-variable distances have been proposed thus far in the literature on metric spaces (see, e.g., Deza and Deza~\cite[Chapter 3]{Deza2014}), the concept of $n$-distance was introduced recently by the authors \cite{KisMarTeh16,KisMarTeh18} as an $n$-ary generalization of both the concepts of distance (i.e., $2$-distance) and $D$-metric \cite{Dhage1992} (i.e., $3$-distance).

One of the interesting features of this new concept is the existence of the associated best constant $K^*_n$, which is generally not easy to compute (see \cite{KisMar,KisMarTeh18}). For instance, we show in this paper that the map that gives the total length of a minimum spanning tree of the complete Euclidean graph constructed from $n$ points in $\R^q$ is an $n$-distance. We also show that the associated best constants for $n=2$ and $n=3$ are given by $K^*_2=1$ and $K^*_3=1/\sqrt{3}$, respectively. However, the exact value of $K^*_n$ remains unknown for any $n\geq 4$.

It was observed \cite{KisMar} that the best constant associated with an $n$-distance is always greater than or equal to $(n-1)^{-1}$. An $n$-distance whose associated best constant has precisely the value $(n-1)^{-1}$ is said to be \emph{standard} \cite{KisMar}; otherwise, it is said to be \emph{nonstandard}. For instance, the cardinality-based $n$-distance defined above is standard.

In this paper we explore various examples of nonstandard $n$-distances based on geometrical constructions and, in some cases, we provide the associated best constants. These examples, as well as many other examples already investigated in \cite{KisMar,KisMarTeh18}, show that when a map $d\colon X^n\to\R_+$ satisfies conditions (i) and (ii) of the definition of an $n$-distance, it is generally not easy to check whether it also satisfies condition (iii). Beyond this difficulty, we also observe that finding the best constant associated with a given $n$-distance is typically a challenging problem, often leading to tricky arguments that strongly depend on the $n$-distance itself. This observation seems to show that this class of problems defines a relatively new direction of research and that general results along this line would be most welcome. We hope that by providing various examples of nonstandard $n$-distances here we might attract researchers and make this exciting topic better known.

This paper is divided into six sections as follows. In Sections 2 to 5, we present and discuss four different examples of $n$-distances; thus these sections can be read independently of each other. In Section 2 we introduce an $n$-distance based on the concept of the Chebyshev ball in any finite-dimensional Euclidean space. We show that it is nonstandard for any $n\geq 3$ and provide the exact value of the associated best constant. In Section 3 we introduce a Euclidean version of this latter $n$-distance, namely the diameter of a largest inner Euclidean ball associated with $n$ points. We also prove that this latter $n$-distance is nonstandard for any $n\geq 3$. However, its associated best constant is not known when $n\geq 4$. Sections 4 and 5 are devoted to two graph-theoretic based $n$-distances. The first one is defined as the length of a minimum spanning tree on the complete Euclidean graph constructed on $n$ points. We show that it is an $n$-distance. It is nonstandard for $n=3$ and $n=4$ and we conjecture that it is nonstandard for any $n\geq 3$. The second one is defined as the length of the Steiner tree on the complete Euclidean graph constructed on $n$ points. We show that it is an $n$-distance and that it is standard for $n=3$. Here again, the exact value of the best constant is not known when $n\geq 4$. Finally, we end this paper in Section 6 with some concluding remarks and open problems.

The \emph{simplex ratio} (\emph{triangle ratio} if $n=2$) $R_d$ associated with an $n$-distance $d$ on $X$ is defined for any $x_1,\ldots,x_n,z\in X$ such that $|\{x_1,\ldots,x_n\}|\geq 2$ by
$$
R_d(x_1,\ldots,x_n;z) ~=~ \frac{d(x_1,\ldots,x_n)}{\sum_{i=1}^nd(x_1,\ldots,x_n)_i^z}{\,}.
$$
We will often use this concept throughout.

\section{Edge length of a largest inner Chebyshev ball}

In this section, we define a nonstandard $n$-distance based on Chebyshev balls in $\R^q$, where $q\geq 1$ is any fixed integer, and we provide the associated best constant. The special case when $q=1$ was already considered by the authors in \cite{KisMar} and can be stated as follows.

\begin{proposition}[Length of a largest inner interval \cite{KisMar}]\label{prop:Llii}
Let $d\colon\R^n\to\R_+$ be the map defined by
$$
d(x_1,\ldots,x_n) ~=~ \max_{i=1,\ldots,n-1}(x_{(i+1)}-x_{(i)}),
$$
where the symbol $x_{(i)}$ stands for the $i$th smallest element among $x_1,\ldots,x_n$. Then $d$ is an $n$-distance on $\R$. Its best constant is $K^*_n=2/n$ and is attained at any $(x_1,\ldots,x_n;z)$ such that $x_1< x_2=\cdots =x_n$ and $z=(x_1+x_2)/2$.
\end{proposition}

Suppose now that $q\geq 2$. Recall that a (closed) Chebyshev $q$-ball in $\R^q$ of radius $r>0$ centered at a point $c\in\R^q$ is the hypercube defined by
$$
B_r[c] ~=~ \{x\in\R^q:\|x-c\|_{\infty}\leq r\},
$$
where $\|\cdot\|_{\infty}$ is the Chebyshev norm in $\R^q$. Each face of this hypercube is parallel to one of the coordinate hyperplanes.

\begin{definition}\label{de:InnCheBall}
Let $q\geq 2$ be an integer and let $x_1,\ldots,x_n$ be $n$ points in $\R^q$. We say that a Chebyshev $q$-ball $B$ in $\R^q$ is an \emph{inner Chebyshev ball associated with $x_1,\ldots,x_n$} if it satisfies the following two conditions.
\begin{itemize}
\item The interior of $B$ contains none of the points $x_1,\ldots,x_n$.
\item At least one point lies on one face of $B$ and at least one point lies on the opposite face.
\end{itemize}
\end{definition}

In the following proposition, we show that the map $d\colon(\R^q)^n\to\R_+$ that carries the $n$-tuple $(x_1,\ldots,x_n)$ into the edge length of a largest inner Chebyshev ball associated with $x_1,\ldots,x_n$ is an $n$-distance on $\R^q$. We also provide its associated best constant. In Figure~\ref{fig:5pq2} we illustrate this $n$-distance through an example based on $5$ points in the plane ($q=2$).

\setlength{\unitlength}{6ex}
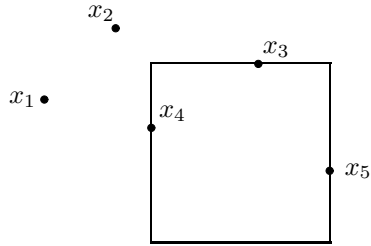
\begin{figure}[htbp]
\begin{center}
\begin{picture}(5,3.5)
\put(2,0){\framebox(2.5,2.5){}}
\put(0.5,2){\circle*{0.1}}\put(1.5,3){\circle*{0.1}}\put(2,1.6){\circle*{0.1}}
\put(3.5,2.5){\circle*{0.1}}\put(4.5,1){\circle*{0.1}}
\put(0.2,2){\makebox(0,0){$x_1$}}\put(1.3,3.2){\makebox(0,0){$x_2$}}
\put(3.75,2.7){\makebox(0,0){$x_3$}}\put(2.3,1.8){\makebox(0,0){$x_4$}}
\put(4.9,1){\makebox(0,0){$x_5$}}
\end{picture}
\caption{A largest inner square constructed from five points in $\R^2$}
\label{fig:5pq2}
\end{center}
\end{figure}

\begin{proposition}\label{propconj:1}
The map $d\colon(\R^q)^n\to\R_+$ that carries $(x_1,\ldots,x_n)$ into the edge length of a largest inner Chebyshev ball associated with $x_1,\ldots,x_n$ is an $n$-distance on $\R^q$. Its best constant is $K^*_n=2/n$ and is attained at any $(x_1,\ldots,x_n;z)$ such that $x_1\neq x_2=\cdots =x_n$ and $z=(x_1+x_2)/2$.
\end{proposition}

\begin{proof}
If $n=2$, then $d$ is the Chebyshev distance on $\R^q$. Now suppose that $n\geq 3$ and let $x_1,\ldots,x_n,z\in\R^q$ be such that $|\{x_1,\ldots,x_n\}|\geq 2$. Modifying the coordinate axes if necessary, we can assume without loss of generality that a largest inner Chebyshev ball associated with the points $x_1,\ldots,x_n$ is given by the set
$$
B_r[0] ~=~ \{x\in\R^q:\|x\|_{\infty}\leq r\}
$$
for some $r>0$ and that $x_1$ lies on one face and $x_2$ lies on the opposite face of the ball. Thus, we have
$$
d(x_1,\ldots,x_n) ~=~ \|x_2-x_1\|_{\infty} ~=~ 2r
$$
and $\|x_i\|_{\infty}\geq r$ for $i=1,\ldots,n$. We then have two exclusive cases to consider.
\begin{itemize}
\item Suppose that $\|z\|_{\infty}<r$. Let $B_1$ (resp.\ $B_2$) be an inner Chebyshev ball associated with $x_1$ and $z$ (resp.\ $x_2$ and $z$), with edge length $\ell_1=\|x_1-z\|_{\infty}$ (resp.\ $\ell_2=\|x_2-z\|_{\infty}$). These balls can always be taken so that $B_1\cup B_2\subset B_r[0]$. We then have $2r =\|x_2-x_1\|_{\infty}\leq \ell_1+\ell_2$ and hence $\max\{\ell_1,\ell_2\}\geq r$. It follows that
$$
d(x_1,\ldots,x_n)_i^z ~\geq ~
\begin{cases}
\ell_2, & \text{if $i=1$},\\
\ell_1, & \text{if $i=2$},\\
\max\{\ell_1,\ell_2\}, & \text{otherwise}.
\end{cases}
$$
Therefore,
$$
\sum_{i=1}^nd(x_1,\ldots,x_n)_i^z ~\geq ~ 2r +(n-2)r ~=~ \frac{n}{2}{\,}d(x_1,\ldots,x_n).
$$
\item Suppose that $\|z\|_{\infty}\geq r$.
Set $x_{n+1}=z$ and let $k\in\{3,\ldots,n+1\}$ be such that $\|x_k\|_{\infty}\leq\|x_i\|_{\infty}$ for $i=3,\ldots,n+1$. Then, like in the previous case, we can see that
$$
d(x_1,\ldots,x_n)_i^z ~\geq ~
\begin{cases}
\|x_k-x_2\|_{\infty}{\,}, & \text{if $i=1$},\\
\|x_k-x_1\|_{\infty}{\,}, & \text{if $i=2$},\\
2r, & \text{otherwise}.
\end{cases}
$$
Therefore,
$$
\sum_{i=1}^nd(x_1,\ldots,x_n)_i^z ~\geq ~ 2r + (n-2)2r ~=~ (n-1){\,}d(x_1,\ldots,x_n).
$$
\end{itemize}
To summarize, we have
$$
K^*_n ~\leq ~ \max\{(n-1)^{-1},2/n\} ~=~ 2/n.
$$
To complete the proof, it suffices to observe that the best constant is attained at any tuple having the claimed properties.
\end{proof}

\section{Diameter of a largest inner Euclidean ball}

The following definition provides a Euclidean version of Definition~\ref{de:InnCheBall}. This definition seems more natural than the concept of the inner Chebyshev ball in the sense that it is independent of the location and inclination of the coordinate system. However, as we will see, it is much more difficult to investigate.

Recall that a (closed) Euclidean $q$-ball in $\R^q$ of radius $r>0$ centered at a point $c\in\R^q$ is defined by the set
$$
B_r[c] ~=~ \{x\in\R^q:|x-c|\leq r\},
$$
where $|\cdot|=\|\cdot\|_2$ is the Euclidean norm in $\R^q$.

\begin{definition}\label{de:DLIEB57}
Let $q\geq 2$ be an integer and let $x_1,\ldots,x_n$ be $n$ points in $\R^q$. We say that a Euclidean $q$-ball $B$ in $\R^q$ is \emph{an inner ball associated with $x_1,\ldots,x_n$} if it satisfies the following two conditions.
\begin{itemize}
\item The interior of $B$ contains none of the points $x_1,\ldots,x_n$.
\item Two of the points $x_1,\ldots,x_n$ are the endpoints of a diameter of $B$.
\end{itemize}
\end{definition}

In this section we show that the map $d\colon(\R^q)^n\to\R_+$ that carries $(x_1,\ldots,x_n)$ into the diameter of a largest inner ball associated with $x_1,\ldots,x_n$ is an $n$-distance on $\R^q$. Figure~\ref{fig:5pq2c} illustrates this map through an example based on $5$ points in the plane ($q=2$). We show that this $n$-distance is nonstandard if and only if $n\geq 3$. We also provide the value of $K^*_3$ and show that $K^*_n> 1/\pi$ for any $n$. The value of $K^*_n$ for any $n\geq 4$ remains unknown and finding this value seems to constitute an interesting open question.

\setlength{\unitlength}{7ex}
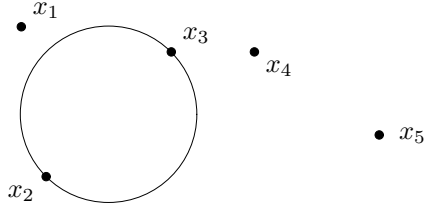
\begin{figure}[htbp]
\begin{center}
\begin{picture}(5.5,3)
\put(1.75,1.25){\circle{2.1213}}
\put(0.7,2.3){\circle*{0.1}}\put(1,0.5){\circle*{0.1}}\put(2.5,2){\circle*{0.1}}
\put(3.5,2){\circle*{0.1}}\put(5,1){\circle*{0.1}}
\put(1,2.5){\makebox(0,0){$x_1$}}\put(0.7,0.3){\makebox(0,0){$x_2$}}
\put(2.8,2.2){\makebox(0,0){$x_3$}}\put(3.8,1.8){\makebox(0,0){$x_4$}}
\put(5.4,1){\makebox(0,0){$x_5$}}
\end{picture}
\caption{A largest inner ball constructed from five points in $\R^2$}
\label{fig:5pq2c}
\end{center}
\end{figure}

Let us explore this problem carefully by first considering the cases when $n=2$ and $n=3$. If $n=2$, then $d$ is the usual Euclidean distance on $\R^q$, and so it is standard.

Let us now consider the case when $n=3$. Any pairwise distinct points $x_1,x_2,x_3\in\R^q$ are the vertices of a (possibly degenerate) triangle. It is then clear that if this triangle is right or acute, then
$$
d(x_1,x_2,x_3) ~=~ \max_{\{i,j\}\subset\{1,2,3\}}|x_i-x_j|.
$$
Now, if the triangle is obtuse, then $d(x_1,x_2,x_3)$ is necessarily the second highest side length of the triangle, that is,
$$
d(x_1,x_2,x_3) ~=~ \mathop{\mathrm{median}}\limits_{\{i,j\}\subset\{1,2,3\}}|x_i-x_j|.
$$
In both cases, it is then geometrically clear that
\begin{equation}\label{eq:65sd6afs}
d(x_1,x_2,x_3) ~\geq ~ \frac{1}{2}{\,}\max_{\{i,j\}\subset\{1,2,3\}}|x_i-x_j|{\,},\qquad x_1,x_2,x_3\in\R^q.
\end{equation}

Now, let $x_1,x_2,x_3,z\in\R^q$ satisfying $|\{x_1,x_2,x_3\}|\geq 2$ and suppose without loss of generality that $x_1$ and $x_2$ are the endpoints of a diameter of a largest inner ball associated with $x_1,x_2,x_3$. Using \eqref{eq:65sd6afs} and then the triangle inequality for $|\cdot|$, we obtain
\begin{eqnarray*}
\sum_{i=1}^3d(x_1,x_2,x_3)_i^z & \geq & \frac{1}{2}{\,}\big(|x_2-x_3|+|x_3-x_1|+|x_1-x_2|\big)\\
&\geq &  |x_1-x_2| ~=~  d(x_1,x_2,x_3),
\end{eqnarray*}
which shows that $d$ is a $3$-distance. To see that it is nonstandard, consider $x_1\neq x_2=x_3$ and $z=(x_1+x_2)/2$. We then obtain $K^*_3\geq 2/3\approx 0.667$.

We actually have the stronger inequality $K^*_3\geq\rho$, where
$$
\rho ~=~ \frac{1}{7}\sqrt{20+2\sqrt{2}} ~\approx ~ 0.683,
$$
which provides a better lower bound for $K^*_3$. Indeed, consider the points $x_1=(-1,0)$, $x_2=(1,0)$, and $x_3=(\sqrt{2}/2,\sqrt{2}/2)$ in $\R^2$. Let also $\varepsilon >0$ be sufficiently small and let $z_{\varepsilon}=(0,\sqrt{2}-1+\varepsilon)$. Thus, the points $x_1,x_2,x_3$ form a right triangle, with $\sphericalangle x_2x_1x_3=\frac{\pi}{8}$, and $z_{\varepsilon}$ is obtained by lifting to a height of $\varepsilon$ the $y$-intercept of the line joining $x_1$ and $x_3$ (see Figure~\ref{fig:cons446}). We then have
\begin{eqnarray*}
2 &=& |x_1-x_2| ~=~ d(x_1,x_2,x_3) ~\leq~ K^*_3{\,}\sum_{i=1}^3d(x_1,x_2,x_3)_i^{z_{\varepsilon}}\\
&\leq & K^*_3{\,}\big(|x_2-x_3|+2{\,}|x_1-z_{\varepsilon}|\big),
\end{eqnarray*}
for small values of $\varepsilon >0$, where
$$
|x_2-x_3|^2 ~=~ 2-\sqrt{2}\qquad\text{and}\qquad |x_1-z_{\varepsilon}|^2 ~=~ 1+(\sqrt{2}-1+\varepsilon)^2.
$$
This shows that $K^*_3\geq\rho$. Moreover, the value $\rho$ is not attained in this example.

\setlength{\unitlength}{4ex}
\begin{figure}[htbp]
\begin{center}
\begin{picture}(6,6)
\put(3,3){\circle{6}}
\put(0,3){\circle*{0.15}}\put(-0.6,3){\makebox(0,0){$x_1$}}
\put(6,3){\circle*{0.15}}\put(6.7,3){\makebox(0,0){$x_2$}}
\put(5.12,5.12){\circle*{0.15}}\put(5.9,5.12){\makebox(0,0){$x_3$}}
\put(3,4.6){\circle*{0.15}}\put(2.8,5.1){\makebox(0,0){$z_{\varepsilon}$}}
\drawline[0](0,3)(3,4.6)(5.12,5.12)(0,3)(6,3)(5.12,5.12)
\end{picture}
\caption{An example involving three points $x_1,x_2,x_3$ of $\R^2$}
\label{fig:cons446}
\end{center}
\end{figure}
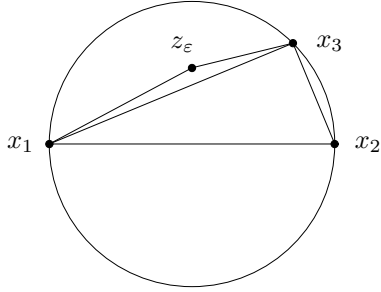

Although the value of $K^*_3$ seems difficult to obtain, most likely due to the fact that the map $d$ is rather discontinuous, we now prove that the value of $K^*_3$ is precisely equal to $\rho$.

\begin{proposition}\label{prop:K3rho}
We have $K^*_3=\rho$ and this value is not attained.
\end{proposition}

\begin{proof}
Let $x_1,x_2,x_3,z\in\R^q$ such that $|\{x_1,x_2,x_3\}|\geq 2$. We can assume without loss of generality that $q=2$ and that $x_1=(-1,0)$ and $x_2=(1,0)$ are the endpoints of a diameter of a largest inner ball associated with $x_1,x_2,x_3$. We then have $d(x_1,x_2,x_3)=|x_1-x_2|$ and $|x_3|\geq 1$.

Using \eqref{eq:65sd6afs} and then the triangle inequality for $|\cdot|$, we have
$$
d(z,x_2,x_3)+d(x_1,z,x_3) ~\geq ~ \frac{1}{2}(|x_2-x_3|+|x_1-x_3|) ~\geq ~\frac{1}{2}{\,}|x_1-x_2|.
$$
If $|z|\geq 1$, then $d(x_1,x_2,z)\geq |x_1-x_2|$ and hence $R_d(x_1,x_2,x_3;z)\leq\frac{2}{3}$. Thus, to ensure that $R_d(x_1,x_2,x_3;z)>\frac{2}{3}$, we have to assume $|z|<1$. In this case, we also have
\begin{equation}\label{eq:7xycv5}
d(x_1,x_2,z) ~=~ \max\{|x_1-z|,|x_2-z|\} ~\geq ~ \frac{1}{2}{\,}|x_1-x_2|.
\end{equation}

Now, let $x_{3,1}$ denote the first coordinate of $x_3$ and let us show that if $|x_{3,1}|\geq 1$, then $R_d(x_1,x_2,x_3;z)\leq\frac{2}{3}$. Suppose that $x_{3,1}\leq -1$ (the other case is similar). We have
\begin{eqnarray*}
d(z,x_2,x_3)+d(x_1,z,x_3) &\geq & |x_2-z| + \min\{|x_1-z|,|x_3-z|\}\\
&\geq & \min\{|x_1-x_2|,|x_2-x_3|\} ~=~ |x_1-x_2|.
\end{eqnarray*}
Combining this result with \eqref{eq:7xycv5}, we see that $R_d(x_1,x_2,x_3;z)\leq\frac{2}{3}$. Thus, to ensure that $R_d(x_1,x_2,x_3;z)>\frac{2}{3}$, we have to assume that $|x_{3,1}|< 1$.

Now, let $x_{3,2}$ (resp.\ $x_{0,2}$) denote the second coordinate of $x_3$ (resp.\ $z$). Let us show that if $x_{3,2}x_{0,2}\leq 0$, then $R_d(x_1,x_2,x_3;z)\leq\frac{2}{3}$. Since $x_1$ and $z$ are the endpoints of a diameter of an inner circle associated with $x_1,z,x_3$, we must have $d(x_1,z,x_3)\geq |x_1-z|$. We show similarly that $d(z,x_2,x_3)\geq |x_2-z|$. We then have
$$
d(z,x_2,x_3)+d(x_1,z,x_3) ~\geq ~ |x_1-x_2|.
$$
Combining this result with \eqref{eq:7xycv5}, we see that $R_d(x_1,x_2,x_3;z)\leq\frac{2}{3}$. Thus, to ensure that $R_d(x_1,x_2,x_3;z)>\frac{2}{3}$, we have to further assume that $x_{3,2}x_{0,2}>0$. Using a symmetry argument, we can assume that $x_{3,2}>0$ and $x_{0,2}>0$.

Now, let us show that if $z$ lies in the (closed) triangle with vertices $x_1,x_2,x_3$, then $R_d(x_1,x_2,x_3;z)\leq\frac{2}{3}$. If $\sphericalangle x_1zx_3\leq\frac{\pi}{2}$, then
$$
d(x_1,z,x_3) ~=~ \max\{|x_1-z|,|x_3-z|,|x_1-x_3|\} ~\geq ~|x_1-z|.
$$
If $\sphericalangle x_1zx_3 > \frac{\pi}{2}$, then
$$
d(x_1,z,x_3) ~=~ \mathrm{median}\{|x_1-z|,|x_3-z|,|x_1-x_3|\} ~\geq ~|x_1-z|.
$$
Similarly, we show that $d(z,x_2,x_3)\geq |x_2-z|$. We then have
$$
d(z,x_2,x_3)+d(x_1,z,x_3) ~\geq ~ |x_1-x_2|.
$$
Combining this result with \eqref{eq:7xycv5}, we see that $R_d(x_1,x_2,x_3;z)\leq\frac{2}{3}$. Thus, to ensure that $R_d(x_1,x_2,x_3;z)>\frac{2}{3}$, we have to further assume that $z$ does not lie in the triangle with vertices $x_1,x_2,x_3$.

Now, let us show that if $|x_3|>1$, then there is $x'_3\in\R^2$, with $|x'_3|=1$, such that $R_d(x_1,x_2,x_3';z)>R_d(x_1,x_2,x_3;z)$. Suppose that $|x_1-z|\leq |x_2-z|$ (the other case is similar) and let $x'_3\in\R^2$ the unique point on the line through $x_1$ and $x_3$ such that $|x'_3|=1$. We observe that $\sphericalangle x_1zx'_3 >\frac{\pi}{2}$ and hence $\sphericalangle x_1zx_3 >\frac{\pi}{2}$, which implies that
$$
d(x_1,z,x_3) ~\geq ~ d(x_1,z,x'_3).
$$
We also observe that $|x'_3-z|<|x_3-z|$ and $|x_2-x'_3|<|x_2-x_3|$ and $\sphericalangle x_2x'_3z >\frac{\pi}{2}$. Thus, if $\sphericalangle zx_3x_2 >\frac{\pi}{2}$, then
\begin{eqnarray*}
d(z,x_2,x_3) &=& \max\{|x_2-x_3|,|x_3-z|\} ~>~ \max\{|x_2-x'_3|,|x'_3-z|\}\\
&=& d(z,x_2,x'_3).
\end{eqnarray*}
If $\sphericalangle zx_3x_2 \leq\frac{\pi}{2}$, then
$$
d(z,x_2,x_3)~\geq ~ |x_2-z| ~>~ \max\{|x_2-x'_3|,|x'_3-z|\} ~=~ d(z,x_2,x'_3).
$$

We then have
$$
d(z,x_2,x_3)+d(x_1,z,x_3) ~>~ d(z,x_2,x'_3)+d(x_1,z,x'_3),
$$
that is, $R_d(x_1,x_2,x_3';z)>R_d(x_1,x_2,x_3;z)$.

Thus, to approach the supremum of $R_d(x_1,x_2,x_3;z)$, we have to assume that $|x_3|=1$. We then have
\begin{multline*}
d(z,x_2,x_3)+d(x_1,z,x_3)+d(x_1,x_2,z) \\
=~ \max\{|x_2-x_3|,|x_3-z|\}+\max\{|x_1-z|,|x_3-z|\}+\max\{|x_1-z|,|x_2-z|\}.
\end{multline*}
Let $z'$ be the orthogonal projection of $z$ onto the line through $x_1$ and $x_3$. We observe that the value of $R_d(x_1,x_2,x_3;z)$ strictly increases as $z$ moves closer to $z'$. It follows that the value $K^*_3$ cannot be attained.

Now, writing $x_3=(\cos\alpha,\sin\alpha)$ for some $\alpha\in\left]0,\pi\right[$ and $z'=tx_1+(1-t)x_3$ for some $t\in\left]0,1\right[$, we have $|x_1-z'|=2(1-t)\cos\frac{\alpha}{2}$, $|x_3-z'|=2t\cos\frac{\alpha}{2}$, $|x_2-x_3|=2\sin\frac{\alpha}{2}$, and
$$
|x_2-z'|^2 ~=~ |x_2-x_3|^2 + |x_3-z'|^2.
$$
It is then a basic calculus exercise to see that the map $(\alpha,t)\mapsto R_d(x_1,x_2,x_3;z')$ has a global maximum at $(\alpha,t)=(\pi/4,\sqrt{2}-1)$ with value $K^*_3=\rho$. This corresponds to $x_3=(\sqrt{2}/2,\sqrt{2}/2)$ and $z'=(0,\sqrt{2}-1)$, that is $|x_1-z'|=|x_2-z'|$.
\end{proof}

We now consider the general case for any $n\geq 2$ and show in a rather tricky way that the map $d$ is an $n$-distance. We first consider a lemma that uses the concept of the minimum spanning tree of a graph. For background on this concept, see, e.g., \cite{Cie05,WuCha04}.

\begin{lemma}\label{lemma:MinSpanTree}
Let $x_1,\ldots,x_n\in\R^q$. Let $T=(V,E)$ be a minimum spanning tree of the complete Euclidean graph whose vertex set is $V=\{x_1,\ldots,x_n\}$. Then every edge in $E$ is a diameter of an inner ball associated with $x_1,\ldots,x_n$.
\end{lemma}

\begin{proof}
For any $i,j\in\{1,\ldots,n\}$, denote the edge $\{x_i,x_j\}$ by $e_{i,j}$. Suppose that there exist $i,j\in\{1,\ldots,n\}$ such that the edge $e_{i,j}$ is in $E$ and is not a diameter of an inner ball associated with $x_1,\ldots,x_n$. This means that there exists $k\in\{1,\ldots,n\}\setminus\{i,j\}$ such that $x_k$ is in the interior of the ball whose $e_{i,j}$ is a diameter. It follows that either
$$
(V,E\cup\{e_{i,k}\}\setminus\{e_{i,j}\})\quad\text{or}\quad (V,E\cup\{e_{j,k}\}\setminus\{e_{i,j}\})
$$
is a spanning tree that is shorter than $T$. This contradicts the definition of $T$.
\end{proof}

\begin{proposition}\label{prop:LIB-nD3}
The map $d\colon(\R^q)^n\to\R_+$ that carries $(x_1,\ldots,x_n)$ into the diameter of a largest inner ball associated with $x_1,\ldots,x_n$ is an $n$-distance on $\R^q$.
\end{proposition}

\begin{proof}
We have seen that the result holds for $n=2$ and $n=3$. We can therefore assume that $n\geq 4$. Let $x_1,\ldots,x_n,z\in\R^q$ satisfying $|\{x_1,\ldots,x_n\}|\geq 2$ and suppose without loss of generality that $x_1$ and $x_2$ are the endpoints of a diameter of a largest inner ball $B$ associated with $x_1,\ldots,x_n$. Thus, we have $d(x_1,\ldots,x_n)=|x_1-x_2|$. We then have two cases to consider.
\begin{itemize}
\item Assume that $z$ is not in the interior of $B$. In this case, we simply have
\begin{eqnarray*}
\sum_{i=1}^nd(x_1,\ldots,x_n)_i^z &\geq & \sum_{i=3}^nd(x_1,\ldots,x_n)_i^z ~\geq ~ (n-2){\,}|x_1-x_2|\\
& \geq & d(x_1,\ldots,x_n).
\end{eqnarray*}
\item Assume that $z$ is in the interior of $B$. Set $x_{n+1}=z$ and let $\mathcal{G}$ be the complete Euclidean graph whose vertex set is $\{x_1,\ldots,x_{n+1}\}$. We can assume that this graph has at least $4$ distinct vertices (otherwise, the problem is trivial). By Lemma~\ref{lemma:MinSpanTree}, there is a simple path $P$ in $\mathcal{G}$ connecting $x_1$ to $x_2$ with the property that every edge in $P$ is a diameter of an inner ball associated with $x_1,\ldots,x_{n+1}$. Let $\ell_1$ be the length of a longest edge in $P$ and let $\ell_3$ be the length of a third longest edge in $P$. Let $j,k\in\{1,\ldots,n+1\}$ such that $\{x_j,x_k\}\in P$ and $|x_j-x_k|=\ell_1$. Then
$$
d(x_1,\ldots,x_n)_i^z ~\geq ~
\begin{cases}
\ell_1, & \text{if $i\notin\{j,k\}$}\\
\ell_3, & \text{otherwise}.
\end{cases}
$$
We then have
\begin{eqnarray*}
\sum_{i=1}^nd(x_1,\ldots,x_n)_i^z &\geq & (n-2)\ell_1+2\ell_3 ~\geq ~ \sum_{e\in P}\ell_e ~> ~ |x_1-x_2|\\
&=& d(x_1,\ldots,x_n),
\end{eqnarray*}
where $\ell_e$ denotes the length of $e$.
\end{itemize}
To summarize, we have shown that the simplex inequality holds in both cases. This completes the proof.
\end{proof}

\begin{remark}
The last two inequalities in the proof of Proposition~\ref{prop:LIB-nD3} are actually rather fine. To illustrate this observation, in the Euclidean plane $\R^2$ take the points $x_1=(-1,0)$, $x_2=(1,0)$, $x_3=(-1,\varepsilon)$, $x_4=(1,\varepsilon)$, and $z=(0,\varepsilon)$ for some $\varepsilon >0$. Then we have
$$
d(x_1,x_2,x_3,x_4)_i^z ~=~ \sqrt{1+\varepsilon^2},\quad i=1,\ldots,4
$$
and
$$
2\ell_1+2\ell_3 ~=~ \sum_{e\in P}\ell_e ~=~ 2(1+\varepsilon) ~>~ 2 ~=~ \|x_1-x_2\|_2.
$$
\end{remark}

As mentioned in the beginning on this section, the best constant $K^*_n$ associated with the $n$-distance defined in Proposition~\ref{prop:LIB-nD3} is not known when $n\geq 4$ and finding its value seems to be a challenging problem. However, the next proposition provides a lower bound for $K^*_n$, which reveals the surprising observation that the inequality $K^*_n>1/\pi$ always holds. This shows in particular that this $n$-distance is nonstandard when $n\geq 3$. For any integer $p\geq 0$, let $T_p$ be the $p$th degree Chebyshev polynomial of the first kind.

\begin{proposition}\label{prop:knLower452}
Let $n\geq 4$ be an integer and let $p=\left\lfloor n/2\right\rfloor -1$. Then, we have
$$
K^*_n ~\geq ~ \frac{1}{n\lambda_n} ~>~ \frac{1}{n\,\sin\frac{\pi}{2p+4}} ~>~ \frac{2p+4}{n\pi} ~>~ \frac{1}{\pi} ~\approx ~0.318{\,},
$$
where $\lambda_n$ is the unique solution to the equation $$T_p(\sqrt{1-x^2}) ~=~ 2x$$ lying in the open interval $]0,\sin\frac{\pi}{2p+4}[$. Moreover, we have
$$
\lim_{n\to\infty}\frac{1}{n\lambda_n} ~=~ \frac{1}{\pi}{\,}.
$$
\end{proposition}

\begin{proof}
Consider the points $x_1=(-1,0)$ and $x_2=(1,0)$ in $\R^2$. Let also $x_3,\ldots,x_n$ placed clockwise on the upper part of the unit circle so that
\begin{itemize}
\item $x_{2i}=x_{2i-1}$ for $2\leq i\leq p$,
\item $x_n=x_{n-1}$ if $n$ is even and $x_n=x_{n-1}=x_{n-2}$ if $n$ is odd,
\item $|x_4-x_1|=\frac{1}{2}{\,}|x_n-x_2|$ and $|x_{2i+2}-x_{2i}|=|x_4-x_1|$ for $2\leq i\leq p$.
\end{itemize}
We then have $p+2$ distinct points. Let also $z=(x_2+x_n)/2$. Figure~\ref{fig:cons446zz} illustrates this construction for $n=9$ (which implies $p=3$).

\setlength{\unitlength}{4ex}
\begin{figure}[htbp]
\begin{center}
\begin{picture}(6,6.2)
\put(3,3){\circle{6}}
\put(0,3){\circle*{0.15}}\put(-0.6,3){\makebox(0,0){$x_1$}}
\put(6,3){\circle*{0.15}}\put(6.7,3){\makebox(0,0){$x_2$}}
\put(0.548,4.73){\circle*{0.15}}\put(-0.6,4.73){\makebox(0,0){$x_3=x_4$}}
\put(1.99,5.83){\circle*{0.15}}\put(0.9,6.1){\makebox(0,0){$x_5=x_6$}}
\put(3.81,5.89){\circle*{0.15}}\put(5.5,6.1){\makebox(0,0){$x_7=x_8=x_9$}}
\put(4.90,4.44){\circle*{0.15}}\put(4.5,4.46){\makebox(0,0){$z$}}
\drawline[0](0,3)(6,3)(3.81,5.89)(3,3)(1.99,5.83)(0.548,4.73)(3,3)
\drawline[0](0,3)(0.548,4.73)
\drawline[0](1.99,5.83)(3.81,5.89)
\end{picture}
\caption{An example with $n=9$}
\label{fig:cons446zz}
\end{center}
\end{figure}
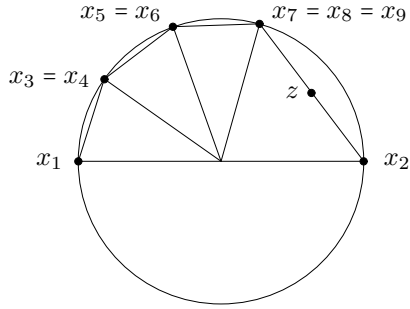

The distances between two consecutive points on the circle can be computed as follows. Setting $\lambda =\frac{1}{2}|x_4-x_1|=\sin\frac{\alpha}{2}$, where $\alpha = \sphericalangle x_10x_4$, we necessarily have $\alpha <\frac{\pi}{p+2}$ and hence $\lambda <\sin\frac{\pi}{2p+4}$. We then obtain
$$
2\lambda  ~=~ \frac{1}{2}{\,}|x_n-x_2| ~=~ \sin\frac{\pi-p\alpha}{2} ~=~ \cos\frac{p\alpha}{2} ~=~ T_p(\cos\frac{\alpha}{2}) ~=~ T_p(\sqrt{1-\lambda^2}).
$$
It is actually a basic exercise to show that the function $f_p(x)=T_p(\sqrt{1-x^2})-2x$ has exactly one zero in the interval $]0,\sin\frac{\pi}{2p+4}[$.

Now, we observe that
$$
2 ~=~ d(x_1,\ldots,x_n) ~\leq ~ K^*_n{\,}\sum_{i=1}^nd(x_1,\ldots,x_n)_i^z ~=~ n K^*_n {\,}2\lambda,
$$
which shows that $K^*_n\geq 1/(n\lambda)$, where $\lambda$ can be replaced with $\lambda_n$ to express its dependency on $n$. The other three inequalities follow trivially.

Let us now prove the last part of the proposition. It is geometrically clear that the length $(\lfloor n/2\rfloor +1)2\lambda_n$ tends to the length $\pi$ of the upper half circle as $n$ increases to infinity. We then have
$$
\lim_{n\to\infty} \frac{1}{n\lambda_n} ~=~ \lim_{n\to\infty} \frac{2\lfloor n/2\rfloor +2}{n\pi} ~=~ \frac{1}{\pi}
$$
and this completes the proof.
\end{proof}

\begin{table}
\begin{center}
$$
\begin{array}{|c|ccccccc|}
\hline n & 4 & 5 & 6 & 10 & 20 & 50 & 80\\
\hline K^*_n~\geq & 0.559 & 0.447 & 0.455 & 0.391 & 0.352 & 0.331 & 0.326 \\
\hline
\end{array}
$$
\medskip
\caption{Values of $1/(n\lambda_n)$ for various $n$}
\label{tab:fval}
\end{center}
\end{table}

Table~\ref{tab:fval} provides the lower bound $1/(n\lambda_n)$ for some values of $n$. We believe that this lower bound could likely be improved. However, as an intermediate challenge, we conjecture that $\lim_{n\to\infty} K_n^*=1/\pi$.

\section{Length of a minimum spanning tree of the complete Euclidean graph}

Let $q\geq 2$ be an integer, let $x_1,\ldots,x_n\in\R^q$, and let $\mathcal{G}(x_1,\ldots,x_n)$ be the complete Euclidean graph whose vertex set is $V=\{x_1,\ldots,x_n\}$. In this section, we show that the map $d\colon (\R^q)^n\to\R_+$ that carries $(x_1,\ldots,x_n)$ into the total length of a minimal spanning tree of $\mathcal{G}(x_1,\ldots,x_n)$ is an $n$-distance on $\R^q$. We also show that $K^*_2=1$, $K^*_3=1/\sqrt{3}$, and $K^*_4\geq\sqrt{2}/4$. The exact value of $K^*_n$ for any $n\geq 4$ remains unknown.

Let us first consider the special cases when $n=2$ and $n=3$. If $n=2$, then $d$ is the usual Euclidean distance on $\R^q$, and so it is standard (i.e., $K^*_2=1$).

Assume now that $n=3$. Let $x_1,x_2,x_3,z\in\R^q$ such that $|\{x_1,x_2,x_3\}|\geq 2$. Setting $\ell_{i,j}=|x_i-x_j|$ for all $i,j$, we then have
$$
d(x_1,x_2,x_3) ~=~ \ell_{1,2}+\ell_{2,3}+\ell_{3,1}-\max\{\ell_{1,2},\ell_{2,3},\ell_{3,1}\}.
$$
It is then clear that $d(x_i,x_j,z)\geq\ell_{i,j}$ for any $\{i,j\}\subset\{1,2,3\}$. It follows that
$$
\sum_{i=1}^3d(x_1,x_2,x_3)_i^z ~\geq ~ \ell_{1,2}+\ell_{2,3}+\ell_{3,1}
$$
and hence $d$ is a $3$-distance. To see that it is nonstandard, just take $x_1,x_2,x_3$ as the vertices of an equilateral triangle and $x_0$ as the centroid. This example shows that $K^*_3\geq 1/\sqrt{3}\approx 0.577$. We now prove that $K^*_3$ is exactly $1/\sqrt{3}$.

Recall that the \emph{Fermat point} of a triangle $ABC$ is the point $F=F(A,B,C)$ that minimizes the distance $|F-A| + |F-B| + |F-C|$. For background, see, e.g., \cite[Chapter II]{BolMarSol99}. The Fermat point can be easily constructed. If there is an angle greater than $2\pi/3$ at some vertex, then that vertex is the Fermat point. Otherwise, draw equilateral triangles on each of the sides of $ABC$. Connect the far vertex of each equilateral triangle to the opposite vertex of triangle $ABC$. Doing this for each of the three equilateral triangles results in a single common point of intersection for all three lines, which is the Fermat point.

\begin{lemma}\label{lemma:msp35t}
For any $x_0, x_1,x_2,x_3\in\R^q$ such that $|\{x_1,x_2,x_3\}|\geq 2$, we have
\begin{equation}\label{eq:msp35t}
R_d(x_1,x_2,x_3;x_0) ~=~ \max_{z\in\R^2} R_d(x_1,x_2,x_3;z)
\end{equation}
if and only if $x_0=F(x_1,x_2,x_3)$.
\end{lemma}

\begin{proof}
(Necessity) Let us first show that $x_0$ must be inside or on the edge of $\triangle x_1x_2x_3$. Suppose that the point $x_0$ that satisfies \eqref{eq:msp35t} is outside $\triangle x_1x_2x_3$ and let $\tilde{x}_0$ be the orthogonal projection of $x_0$ onto $\triangle x_1x_2x_3$. Then we have $|x_i-x_0|>|x_i-\tilde{x}_0|$ for $i=1,2,3$ and hence $d(x_0,x_i,x_j)>d(\tilde{x}_0,x_i,x_j)$ for $\{i,j\}\subset\{1,2,3\}$. If follows that $R_d(x_1,x_2,x_3;x_0)<R_d(x_1,x_2,x_3;\tilde{x}_0)$, a contradiction.

Let us now prove that $x_0=F(x_1,x_2,x_3)$. There are two exclusive cases to consider.
\begin{enumerate}
\item[(a)] Suppose that $\max\{\ell_{i,j},\ell_{0,i},\ell_{0,j}\}=\ell_{i,j}$ for $\{i,j\}\subset\{1,2,3\}$. Then
    $$
    \sum_{i=1}^3d(x_1,x_2,x_3)_i^{x_0}=2(\ell_{0,1}+\ell_{0,2}+\ell_{0,3}),
    $$
    which is minimized only for $x_0=F(x_1,x_2,x_3)$.
\item[(b)] Suppose that $\max\{\ell_{1,2},\ell_{0,1},\ell_{0,2}\}>\ell_{1,2}$, which implies that $\sphericalangle x_1x_0x_2 <\pi/2$. We then have $\sphericalangle x_2x_0x_3\geq\pi/2$ and $\sphericalangle x_3x_0x_1\geq\pi/2$ and hence
    $$
    \max\{\ell_{i,j},\ell_{0,i},\ell_{0,j}\} ~=~\ell_{i,j}{\,},\quad\text{for $\{i,j\}\in\{\{2,3\},\{3,1\}\}$}.
    $$
    We then have
    \begin{eqnarray*}
    \sum_{i=1}^3d(x_1,x_2,x_3)_i^{x_0} &=& (\ell_{0,1}+\ell_{0,2}+\ell_{0,3})+(\ell_{1,2}+\min\{\ell_{0,1},\ell_{0,2}\}+\ell_{0,3})\\
    &=& (\ell_{0,1}+\ell_{0,2}+\ell_{0,3})+w(x_0,x_1,x_2,x_3),
    \end{eqnarray*}
    where $w(x_0,x_1,x_2,x_3)$ is the total length of a minimum spanning tree of the graph $\mathcal{G}(x_0,x_1,x_2,x_3)$. We know that the sum $\ell_{0,1}+\ell_{0,2}+\ell_{0,3}$ is minimized only for $x_0=F(x_1,x_2,x_3)$. Now, considering the classical Euclidean Steiner tree for $x_1,x_2,x_3$, it is known (see, e.g., \cite[p.~4]{Cie98}) that the corresponding Steiner point is $F(x_1,x_2,x_3)$, which means that $w(x_0,x_1,x_2,x_3)$ is minimized only for $x_0=F(x_1,x_2,x_3)$.
\end{enumerate}

(Sufficiency) This results from existence and uniqueness of the Fermat point.
\end{proof}

\begin{proposition}\label{prop:MST333tz}
The map $d\colon(\R^q)^3\to\R$ that carries $(x_1,x_2,x_3)$ into the total length of a minimal spanning tree of $\mathcal{G}(x_1,x_2,x_3)$ is a $3$-distance on $\R^q$. Its best constant is $K_3^*=1/\sqrt{3}$ and is attained at any $(x_1,x_2,x_3;z)$ such that $\triangle x_1x_2x_3$ is an equilateral triangle with centroid $z$.
\end{proposition}

\begin{proof}
Let $x_1,x_2,x_3,z\in\R^2$ be such that $|\{x_1,x_2,x_3\}|\geq 2$. By Lemma~\ref{lemma:msp35t}, the unique optimal choice for $z$ is $F=F(x_1,x_2,x_3)$ and hence we can assume that $q=2$. We can also assume that $F\notin\{x_1,x_2,x_3\}$, for otherwise the simplex ratio $R_d(x_1,x_2,x_3;F)$ would be $\frac{1}{2}$. We now prove that $\triangle ABC$ can always be transformed into an isosceles triangle whose simplex ratio is higher than $R_d(x_1,x_2,x_3;F)$. We can assume without loss of generality that $\ell_{1,2}=\max\{\ell_{1,2},\ell_{2,3},\ell_{3,1}\}$. Let $x'_1,x'_2\in\R^2$ be such that $|x_1-x_3|=|x'_1-x_3|$, $|x_2-x_3|=|x'_2-x_3|$, and $\max\{|x'_1-x_3|,|x'_2-x_3|\}=|x'_1-x'_2|$. Thus, $\triangle x'_1x'_2x_3$ is isosceles. If we assume further that $F\in\triangle x'_1x'_2x_3$, which is always possible, then it is not difficult to see that
$$
|x_i-F| ~\geq ~ |x'_i-F|\quad\text{for $i=1,2,$}
$$
and hence that
\begin{eqnarray*}
|x_1-F|+|x_2-F|+|x_3-F| & \geq & |x'_1-F|+|x'_2-F|+|x_3-F|\\
&\geq & |x'_1-F'|+|x'_2-F'|+|x_3-F'|,
\end{eqnarray*}
where $F'=F(x'_1,x'_2,x_3)$. It follows that
\begin{eqnarray*}
R_d(x_1,x_2,x_3;F) &=& \frac{|x_1-x_3|+|x_2-x_3|}{2(|x_1-F|+|x_2-F|+|x_3-F|)}\\
&\leq & \frac{|x'_1-x_3|+|x'_2-x_3|}{2(|x'_1-F'|+|x'_2-F'|+|x_3-F'|)} ~=~ R_d(x'_1,x'_2,x_3;F').
\end{eqnarray*}
To conclude the proof, we now show that from among all isosceles triangles, only the equilateral one reaches the highest simplex ratio. Fix two distinct points $x_1$ and $x_2$ and let the point $x_3$ vary so that $\triangle x_1x_2x_3$ is always an isosceles triangle with $|x_1-x_3|=|x_2-x_3|$. Again, we can assume that $F\notin\{x_1,x_2,x_3\}$. We can also assume that $|x_1-F|=|x_2-F|=1$. We then have $|x_1-x_2|=\sqrt{3}$. We now set $t=|x_3-F|$ and determine the value of $t$ that maximizes the simplex ratio $f(t)=R_d(x_1,x_2,x_3;F)$. Setting $a(t)=|x_1-x_3|=|x_2-x_3|$, the law of cosines provides the identity $a(t)^2=t^2+t+1$. We then have
$$
f(t) ~=~
\begin{cases}
\frac{\sqrt{t^2+t+1}+\sqrt{3}}{2(t+2)_{\mathstrut}}{\,}, & \text{if $t\geq 1$},\\
\frac{\sqrt{t^2+t+1}^{\mathstrut}}{t+2}{\,}, & \text{if $t\leq 1$}.
\end{cases}
$$
We immediately see that this function has a unique global maximum at $t=1$, which corresponds to an equilateral triangle. Thus, $K^*_3=f(1)=1/\sqrt{3}$.
\end{proof}

We now consider the general case for any $n\geq 2$ and show that the map $d$ is an $n$-distance. The proof uses a similar argument as in the proof of Proposition~\ref{prop:LIB-nD3}.

\begin{proposition}\label{prop:sda7sds33}
The map $d\colon (\R^q)^n\to\R_+$ that carries $(x_1,\ldots,x_n)$ into the total length of a minimal spanning tree of $\mathcal{G}(x_1,\ldots,x_n)$ is an $n$-distance on $\R^q$.
\end{proposition}

\begin{proof}
We have seen that the result holds for $n=2$ and $n=3$. We can therefore assume that $n\geq 4$. Let $x_1,\ldots,x_n,z\in\R^q$ satisfying $|\{x_1,\ldots,x_n\}|\geq 2$. Let also $T=(V,E)$ be a minimal spanning tree of $\mathcal{G}(x_1,\ldots,x_n)$.

Let $\ell$ be the length of the longest edge in $T$ and suppose without loss of generality that $\ell=\ell_{1,2}$. We then have $d(x_1,\ldots,x_n)_i^z\geq \ell$ for $i=3,\ldots,n$. Now, let $T_1$ and $T_2$ be minimal spanning trees of $\mathcal{G}(z,x_2,x_3,\ldots,x_n)$ and $\mathcal{G}(x_1,z,x_3,\ldots,x_n)$, respectively. The union of these trees contains a path from $x_1$ to $x_2$ through $z$. By definition of $\ell$, the length of this path is at least $\ell$. It follows that
$$
d(x_1,\ldots,x_n)_1^z+d(x_1,\ldots,x_n)_2^z ~\geq ~ \ell.
$$
In total, we obtain
$$
\sum_{i=1}^n d(x_1,\ldots,x_n)_i^z ~\geq ~ \ell +(n-2)\ell ~=~ (n-1)\ell ~\geq ~ \sum_{e\in E}\ell_e ~\geq ~ d(x_1,\ldots,x_n),
$$
where $\ell_e$ denotes the length of $e$.
\end{proof}

Finding the exact value of $K^*_n$ for $n\geq 4$ remains an interesting open question. Proposition~\ref{prop:MST333tz} suggests that $K^*_n$ could be attained by considering a regular $n$-gon with vertices $x_1,\ldots,x_n$ and centroid $z$. For $n=4$, this provides the inequality $K^*_4\geq\sqrt{2}/4$. However, for $n\geq 5$, the corresponding simplex ratio is lower than $(n-1)^{-1}$ and hence not useful. Nevertheless, we conjecture that $d$ is nonstandard if and only if $n\geq 3$.

We end this section with the following proposition, which provides both an alternative proof of Proposition~\ref{prop:sda7sds33} and an upper bound for the best constant $K^*_n$.

\begin{proposition}\label{prop:sda7sds332}
The map $d\colon (\R^q)^n\to\R_+$ that carries $(x_1,\ldots,x_n)$ into the total length of a minimal spanning tree of $\mathcal{G}(x_1,\ldots,x_n)$ is an $n$-distance on $\R^q$. Moreover, we have $K^*_n\leq\frac{2}{n}$ and this inequality is strict whenever $n\geq 3$.
\end{proposition}

\begin{proof}
We can assume that $n\geq 3$. Let $x_1,\ldots,x_n,z\in\R^q$ satisfying $|\{x_1,\ldots,x_n\}|\geq 2$. Let $T=(V,E)$ be a minimal spanning tree of $\mathcal{G}(x_1,\ldots,x_n)$. Let also $T_1$ and $T_2$ be minimal spanning trees of $\mathcal{G}(z,x_2,x_3,\ldots,x_n)$ and $\mathcal{G}(x_1,z,x_3,\ldots,x_n)$, respectively. Clearly, $T_1\cup T_2$ is a connected graph with $x_1,\ldots,x_n,z$ as vertices. By definition of $T$, the sum of the lengths of $T_1$ and $T_2$ is always greater than or equal to the length of $T$. That is,
$$
d(x_1,x_2,\ldots,x_n) ~\leq ~ d(z,x_2,\ldots,x_n)+d(x_1,z,\ldots,x_n).
$$
This immediately shows that $d$ is an $n$-distance.

Now, proceeding similarly for the pairs $(T_2,T_3),(T_3,T_4),\ldots,(T_n,T_1)$, and then adding the resulting inequalities, we finally obtain
$$
n{\,}d(x_1,\ldots,x_n) ~\leq ~ 2\,\sum_{i=1}^nd(x_1,\ldots,x_n)_i^z{\,},
$$
from which we immediately derive $K^*_n\leq\frac{2}{n}${\,}. It is then easy to see that this inequality is strict whenever $n\geq 3$.
\end{proof}

\section{Total length of the Euclidean Steiner tree}

The Euclidean Steiner tree problem is a variant of the Euclidean minimum spanning tree problem that can be described as follows (see, e.g., \cite{BraGraThoZach14,Cie98}).

Given $n$ points $x_1,\ldots,x_n$ in the plane $\R^2$, the problem consists in finding the shortest network (in the Euclidean sense) connecting the points, allowing the addition of auxiliary points (called Steiner points) to the set with the purpose of minimizing the total length.

A very simple argument can be used to establish that the length of this shortest network is an $n$-distance.

\begin{proposition}\label{prop:dsa85sfdsd}
The map $d\colon(\R^2)^n\to\R_+$ that carries $(x_1,\ldots,x_n)$ into the length of the Euclidean Steiner tree constructed on the points $x_1,\ldots,x_n$ is an $n$-distance on $\R^2$.
\end{proposition}

\begin{proof}
Clearly, it is enough to establish the simplex inequality. Let $S$ denote the Steiner tree constructed on $x_1,\ldots,x_n$. For any $z\in\R^2$ and any $i\in\{1,\ldots,n\}$, let $S_i$ denote the Steiner tree constructed on $x_1,\ldots,x_{i-1},z,x_{i+1},\ldots,x_n$. Then $\bigcup_iS_i$ is a connected graph containing the points $x_1,\ldots,x_n$. It follows that the sum of the lengths of the $S_i$'s cannot be lower than the length of $S$.
\end{proof}

For $n=3$, there is only one Steiner point, and it is known and easily seen that this point is precisely the Fermat point. We then see that the map $d\colon(\R^2)^3\to\R$ that carries $(x_1,x_2,x_3)$ into the length of the Euclidean Steiner tree constructed on the points $x_1,x_2,x_3$ is defined by
$$
d(x_1,x_2,x_3) ~=~ \min_{x\in\R^2}\sum_{i=1}^3|x_i-x|{\,}.
$$
It was shown \cite{KisMarTeh18} that this map is a $3$-distance with the property that $\frac{1}{2}\leq K^*_3\leq \frac{8}{15}\approx 0.533$. We now show that $K^*_3=\frac{1}{2}$, which means that this $3$-distance is standard.

\begin{proposition}
We have $K_3^*=\frac{1}{2}$ and this value is attained.
\end{proposition}

\begin{proof}
Let $x_1,x_2,x_3,z\in\R^2$. We can assume without loss of generality that the points $x_1,x_2,x_3$ are not collinear so that they form a nondegenerate triangle $T$. Suppose first that $z$ lies in the triangle $T$, including the sides. Using repeatedly the triangle inequality, it is geometrically clear that
$$
\sum_{i=1}^3d(x_1,x_2,x_3)_i^z ~\geq ~ 2{\,}d(x_1,x_2,x_3),
$$
that is, $R_d(x_1,x_2,x_3;z)\leq\frac{1}{2}$. Suppose now that $z$ lies outside $T$ and let $z^*\in\R^2$ be the closest point to $z$ that lies in $T$. The point $z^*$ exists and is unique by convexity of $T$. It is then also geometrically clear that
$$
\sum_{i=1}^3d(x_1,x_2,x_3)_i^z ~\geq ~ \sum_{i=1}^3d(x_1,x_2,x_3)_i^{z^*} ~\geq ~ 2{\,}d(x_1,x_2,x_3),
$$
so $R_d(x_1,x_2,x_3;z)\leq\frac{1}{2}$ again. Thus, we must have $K_3^*=\frac{1}{2}$. This value is attained if $T$ is an equilateral triangle with centroid $z$.
\end{proof}

In the following proposition, we provide both an alternative proof of Proposition~\ref{prop:dsa85sfdsd} and an upper bound for the best constant $K^*_n$.

\begin{proposition}\label{prop:sda7sds332443}
The map $d\colon (\R^q)^n\to\R_+$ that carries $(x_1,\ldots,x_n)$ into the length of the Euclidean Steiner tree constructed on the points $x_1,\ldots,x_n$ is an $n$-distance on $\R^q$. Moreover, we have $K^*_n\leq\frac{2}{n}$ and this inequality is strict whenever $n\geq 3$.
\end{proposition}

\begin{proof}
The proof can be immediately adapted from that of Proposition~\ref{prop:sda7sds332}.
\end{proof}

\section{Concluding remarks}

We have explored four examples of $n$-distances based on geometric constructions. These examples share the common feature that proving the simplex inequality is not immediate and that finding the value of the associated best constant is rather challenging.

It is likely that there are many other interesting examples that can be considered and examined. Actually, there must be plenty of natural ways to construct maps that look like $n$-distances. However, for many of them it may be very tricky to establish that they are genuine $n$-distances and find their associated best constants.

For example, considering the planar version (i.e., $q=2$) of the map introduced in Definition~\ref{de:DLIEB57}, we could relax the second condition for a circle to be inner by considering the following property.
\begin{itemize}
\item \emph{The circle goes through at least two points $x_i$ and $x_j$ for which $\|x_i-x_j\|_2$ is the diameter of the circle or it goes through at least three pairwise distinct points $x_i,x_j,x_k$ that are the vertices of an acute triangle.}\footnote{We consider acute triangles only to avoid highly discontinuous situations (e.g., when considering the three points $x_1=(-1,0)$, $x_2=(0,0)$, $x_3=(1,\varepsilon)$ in $\R^2$ for a small value of $\varepsilon$).}
\end{itemize}
This new condition may provide a better intuition of what a circle ``bounded'' by $n$ points should look like, but it might make the problem more difficult to solve, even when $n=3$.

In conclusion, we observe that more general results on how to prove the simplex inequality and find the exact value of the associated best constant will be most welcome. This area of investigation seems to be very intriguing and our hope, after examining some interesting examples in this paper, is to spark the interest and enthusiasm of researchers in this theory.

Let us end this paper with a few interesting examples that constitute very natural open problems.

\begin{itemize}
\item \emph{Number of lines defined by $n$ points in the Euclidean space.} Let $q\geq 2$ be an integer. Any two distinct points $x,y\in\R^q$ define a unique straight line through these points. It was shown \cite{KisMarTeh18} that the map $d\colon(\R^q)^n\to\R$ that carries $(x_1,\ldots,x_n)$ into the number of distinct lines defined by the points $x_1,\ldots,x_n$ is an $n$-distance on $\R^q$. Also, the associated best constant $K^*_n$ satisfies the inequalities
$$
(n-2+2/n)^{-1} ~\leq ~ K^*_n ~<~ (n-2)^{-1},
$$
where the upper bound $(n-2)^{-1}$ is defined only when $n\geq 3$. After examining this question carefully, we conjecture that $K^*_n=(n-2+2/n)^{-1}$ for any integer $n\geq 2$, which would imply that $d$ is nonstandard if and only if $n\geq 3$. This value is attained, e.g., when $x_1,\ldots,x_n$ are pairwise distinct and placed clockwise on a circle and $z=x_1$.

\item \emph{Diameter of the smallest enclosing ball in the Euclidean space.} Let $q\geq 2$ be an integer. For any points $x_1,\ldots,x_n$ in $\R^q$, we let $S(x_1,\ldots,x_n)$ denote the smallest $(q-1)$-dimensional sphere enclosing the points $x_1,\ldots,x_n$. We conjecture that the map $d\colon(\R^q)^n\to\R$ that carries $(x_1,\ldots,x_n)$ into the diameter (or equivalently the radius) of the sphere $S(x_1,\ldots,x_n)$ is a standard $n$-distance. This statement was actually proved in \cite{KisMarTeh18} when $q=2$. The value $K_n^*=(n-1)^{-1}$ is attained, e.g., when $x_1\neq x_2=\cdots =x_n=z$.

\item \emph{Volume of the smallest enclosing ball in the Euclidean space.} Consider the smallest enclosing ball defined in the previous example and assume that $n\geq 3-2^{1-q}$. We conjecture that the map $d\colon(\R^q)^n\to\R$ that carries $(x_1,\ldots,x_n)$ into the $q$-dimensional volume of the ball bounded by the sphere $S(x_1,\ldots,x_n)$ is an $n$-distance for which $K_n^*=(n-2+2^{1-q})^{-1}$. This statement was proved in \cite{KisMarTeh18} when $q=2$. Moreover, this value of $K_n^*$ is attained, e.g., when $x_1\neq x_2$ and $x_3=\cdots =x_n=z$ is the midpoint of $x_1$ and $x_2$.
\end{itemize}

\section*{Acknowledgments}

G. Kiss is supported by Premium Postdoctoral Fellowship of the Hungarian Academy of Sciences and by the Hungarian National Foundation for Scientific Research, Grant No. K124749.

\end{document}